\documentclass[preprint,review,10pt]{amsart}
\usepackage{amsmath,amssymb,amsfonts,xspace}
\newtheorem{theorem}{Theorem}[section]

\theoremstyle{definition}
\newtheorem{definition}[theorem]{Definition}
\newtheorem{example}[theorem]{Example}

\newtheorem{corollary}[theorem]{Corollary}
\newtheorem{lem}[theorem]{Lemma}

\theoremstyle{remark}

\numberwithin{equation}{section}

\begin{document}
\title{2-prime hyperideals }

\author{Mahdi Anbarloei}
\address{Department of Mathematics, Faculty of Sciences,
Imam Khomeini International University, Qazvin, Iran.
}

\email{m.anbarloei@sci.ikiu.ac.ir}





\begin{abstract}
Multiplicative hyperrings are an important class of algebraic hyperstructures which generalize rings further to allow multiple output values for the multiplication  operation. Let $R$ be a commutative  multiplicative hyperring. A proper hyperideal $I$ of $R$ is called 2-prime if $x \circ y \subseteq I$ for some $x, y \in R$, then $x^2 \subseteq I$ or $y^2 \subseteq I$.   The 2-prime hyperideals  are  a generalization of prime hyperideals. In this paper, we aim to  study 2-prime hyperideals and give some results. Moreover, we investigate $\delta$-2-primary hyperideals which are an expansion of   2-prime hyperideals. 
\end{abstract}
\maketitle
\section{Introduction}
The  theory of algebraic hyperstructures was first introduced  by Marty \cite{marty}. He defined the hypergroups as a generalization of groups  in 1934. Since then, algebraic hyperstructures have been investigated
by many researchers with numerous applications in both pure and applied sciences \cite{f1,f2,f3,f4,f5,f6,f7,f8,f9}. In a classical algebraic structure, the composition of two elements is an element, while in an algebraic hyperstructure, the composition of two elements is a set. Similar to hypergroups, hyperrings are algebraic structures more general than rings, subsitutiting both or only one of the binary operations of addition and multiplication by hyperoperations. The hyperrings
were introduced and studied by many authors  \cite{ameri3, ameri4, f12, f13}. Krasner introduced a type of the hyperring  where
addition is a hyperoperation and multiplication is an ordinary binary operation. Such a hyperring is
called a Krasner hyperring \cite{f11}. Mirvakili and Davvaz introduced  $(m,n)$-hyperrings in \cite{f17} and they defined Krasner $(m,n)$-hyperrings as a subclass of $(m,n)$-hyperrings and as a generalization of Krasner hyperrings in \cite{f18}.
A well-known
type of a hyperring, called the multiplicative hyperring. The hyperring was introduced by Rota in 1982 which the multiplication is a hyperoperation,
while the addition is an operation \cite{f14}. There exists a
general type of hyperrings that both the addition and multiplication are hyperoperations \cite{f15}.  Ameri and Kordi have studied Von Neumann regularity in multiplicative hyperrings \cite{ameri5}. Moreover, they introduced the concept of clean multiplicative hyperrings and studied  some topological
concepts to realize clean elements of a multiplicative hyperring by clopen subsets
of its Zariski topology\cite{ameri6}. The
notions such as (weak)zero divisor, (weak)nilpotent and unit in an arbitrary commutative hyperrings were introduced in \cite{ameri2}. Some equivalence relations - called fundamental relations - play important
roles in the the theory of algebraic hyperstructures. The fundamental
relations are one of the most important and interesting concepts in algebraic
hyperstructures that ordinary algebraic structures are derived from algebraic
hyperstructures by them. For more details about hyperrings and fundamental relations on
hyperrings see \cite{x4, f6, x2, x3, x1, marty, x5, f15}.
Prime ideals and primary ideals are two of the most important structures in commutative algebra. The notion of primeness of hyperideal in a multiplicative hyperring was conceptualized by Procesi and rota in \cite{f16}. Dasgupta extended the prime and primary hyperideals in multiplicative hyperrings in \cite{das}. 
Beddani and Messirdi \cite{bed} introduced a generalization of prime ideals called
2-prime ideals and this idea is further generalized by Ulucak and et. al.  \cite{gu}.
In \cite{don}, Dongsheng  defined a new notion which is called $\delta$-primary ideals in commutative rings. In \cite{ul}, Ulucak introduced the concepts $\delta$-primary and 2-absorbing $\delta$-primary hyperideal over multiplicative hyperrings. In \cite{anb}, we investigate $\delta$-primary hyperideals in
a Krasner $(m, n)$-hyperring which unify prime hyperideals and primary hyperideals.

In this paper we consider the class of multiplicative hyperring as a hyperstructure
$(R,+,\circ)$, where $(R,+)$ is an abelian group, $(R,\circ)$ is a semihypergroup
and the hyperoperation "$\circ$" is distributive with respect to the operation
$" + "$. In this paper we introduce and study the notion of 2-prime hyperideals of multiplicative hyperrings  which are  a generalization of prime hyperideals. Several properties of them are provided. Moreover, we investigate $\delta$-2-primary hyperideals which are an expansion of   2-prime hyperideals.

\section{Preliminaries}
In this section we give some defnitions and results  which we
need to develop our paper.

A hyperoperation "$\circ $" on nonempty set $G$ is a mapping of $G \times G$ into the family of all nonempty subsets of $G$. Assume that "$\circ $" is a hyperoperation on $G$. Then $(G,\circ)$ is called hypergroupoid. The hyperoperation on $G$ can be extended to  subsets of $G$ as follows. Let $X,Y$ be subsets of $G$ and $g \in G$, then 
\[X \circ Y =\cup_{x \in X, y \in Y}x \circ y, \ \  X \circ g=X \circ \{g\}. \]
A hypergroupoid $(G, \circ)$ is called  a semihypergroup if for all $x,y,z \in G$, $(x \circ y) \circ z=x \circ (y \circ z)$, which is associative. A semihypergroup is said to be a hypergroup if  $g \circ G=G=G \circ g$ for all  $g \in G$. A nonempty subset $H$ of a semihypergroup $(G,\circ)$ is called a
subhypergroup if for all $x \in H$ we have $x \circ H=H=H \circ x$. A commutative hypergroup $(G,\circ)$ is canonical if
\begin{itemize}
\item[\rm{(i)}]~there exists  $e \in G$ with $e \circ x=\{x\}$, for every $x \in G$.
\item[\rm{(ii)}]~for every $x \in G$ there exists a unique $x^{-1} \in G$ with $e \in x \circ x^{-1}$.
\item[\rm{(iii)}]~$x \in y  \circ z $ implies $y \in x \circ z^{-1}$.
\end{itemize}
A nonempty set $R$ with two hyperoperations $"+"$ and $"\circ"$ is called  a
hyperring if $(R,+)$ is a canonical hypergroup , $(R,\circ)$ is a semihypergroup with
$r \circ 0=0 \circ r=0$ for all $r \in R$  and the hyperoperation $"\circ"$ is distributive with respect to $+$, i.e., $x \circ (y+z)=x \circ y+x \circ z$ and  $(x+y) \circ z=x \circ z+y \circ z$ for all $x,y,z \in R$. 

\begin{definition} \cite{f10}
A {\it multiplicative hyperring} is an abelian group $(R,+)$ in which a hyperoperation $\circ $ is defined satisfying the following: 
\begin{itemize}
\item[\rm(i)]~ for all $a, b, c \in R$, we have $a \circ (b \circ c)=(a \circ b) \circ c$;
\item[\rm(ii)]~for all $a, b, c \in R$, we have $a\circ (b+c) \subseteq a\circ b+a\circ c$ and $(b+c)\circ a \subseteq b\circ a+c\circ a$;
\item[\rm(iii)]~for all $a, b \in R$, we have $a\circ (-b) = (-a)\circ b = -(a\circ b)$.
\end{itemize}
\end{definition}
If in (ii) the equality holds then we say that the multiplicative hyperring is strongly distributive.

A non empty subset $I$ of a multiplicative hyperring $R$ is a {\it hyperideal} if
\begin{itemize}
\item[\rm(i)]~ If $a, b \in I$, then $a - b \in I$;

\item[\rm(iii)]~ If $x \in I $ and $r \in R$, then $rox \subseteq I$.
\end{itemize}
Let $(\mathbb{Z},+,\cdot)$ be the ring of integers. Corresponding to every subset $A \in P^\star(\mathbb{Z})$ such that $\vert A\vert \geq 2$, there exists a multiplicative hyperring $(\mathbb{Z}_A,+,\circ)$ with $\mathbb{Z}_A=\mathbb{Z}$ and for any $a,b\in \mathbb{Z}_A$, $a \circ b =\{a.r.b\ \vert \ r \in A\}$. 
\begin{definition} \cite{das} 
A proper hyperideal $P$ of $R$ is called a {\it prime hyperideal} if $x\circ y \subseteq P$ for $x,y \in R$ implies that $x \in P$ or $y \in P$. The intersection of all prime hyperideals of $R$ containing $I$ is called the prime radical of $I$, being denoted by $\sqrt{I}$. If the multiplicative hyperring $R$ does not have any prime hyperideal containing $I$, we define $\sqrt{I}=R$. 
\end{definition}
\begin{definition} \cite{ameri}
A proper hyperideal $I$ of $R$ is {\it maximal} in R if for
any hyperideal $J$ of $R$ with $I \subseteq J \subseteq R$ then $J = R$. Also, we say that $R$ is a local multiplicative hyperring, if it has just one maximal hyperideal.
\end{definition}
Let {\bf C} be the class of all finite products of elements of $R$ i.e. ${\bf C} = \{r_1 \circ r_2 \circ . . . \circ r_n \ : \ r_i \in R, n \in \mathbb{N}\} \subseteq P^{\ast }(R)$. A hyperideal $I$ of $R$ is said to be a {\bf C}-hyperideal of $R$ if, for any $A \in {\bf C}, A \cap I \neq \varnothing $ implies $A \subseteq I$.
Let I be a hyperideal of $R$. Then, $D \subseteq \sqrt{I}$ where $D = \{r \in R: r^n \subseteq I \ for \ some \ n \in \mathbb{N}\}$. The equality holds when $I$ is a {\bf C}-hyperideal of $R$(\cite {das}, proposition 3.2). In this paper, we assume that all hyperideals are {\bf C}-hyperideal.
\begin{definition} \cite{das}
A nonzero proper hyperideal $Q$ of $R$ is called a {\it primary hyperideal} if $x\circ y \subseteq Q$ for $x,y \in R$ implies that $x \in Q$ or $y \in \sqrt{Q}$. Since $\sqrt{Q}=P$ is a prime hyperideal of $R$ by Propodition 3.6 in \cite{das}, $Q$ is referred to as a P-primary hyperideal of $R$.
\end{definition}
\begin{definition} 
\cite{ameri} Let $R$ be commutative multiplicative hyperring and $e$ be an identity (i. e., for all $a \in R$, $a \in a\circ e$). An element $x \in R$ in is called {\it unit}, if there exists $y \in R$, such that $e \in x\circ y$.
\end{definition} 
\begin{definition} 
A hyperring $R$ is called an {\it integral hyperdomain}, if for all $x, y \in R$,
$0 \in x . y$ implies that $x = 0$ or $y = 0$. 
\end{definition}
\begin{definition} \cite{ameri}
Let $R$ be a multiplicative hyperring. The element $x \in R$ is
an idempotent if $x \in x^2$.
\end{definition} 
\begin{definition} \cite{ameri2}
An element $a \in R$ is said to be zero divizor if there exists $0 \neq b \in R$ such that $\{0\}=a \circ b.$
\end{definition}
\begin{definition} \cite{f10} 
Let $(R_1, +_1, \circ _1)$ and $(R_2, +_2, \circ_2)$ be two multiplicative hyperrings. A mapping $f$ from
$R_1$ into $R_2$ is said to be a {\it good homomorphism} if for all $x,y \in R_1$, $f(x +_1 y) =f(x)+_2 f(y)$ and $f(x\circ_1y) = f(x)\circ_2 f(y)$.
\end{definition}
\begin{definition} \cite{Ay}
A function $\delta$ is called a hyperideal expansion of $R$ if it assigns to each hyperideal $I$ of $R$ a hyperideal $\delta(I)$ such that it has the following properties: 
\begin{itemize}
\item[\rm{(1)}]~ $I \subseteq \delta(I)$
\item[\rm{(2)}]~if $I \subseteq J$ for any hyperideals $I, J$ of $R$, then $ \delta(I) \subseteq \delta(J)$.
\end{itemize}
\end{definition}
For example, consider the hyperideal expansions $\delta_0$,  $\delta_1$, $\delta_+$ and $\delta_*$ of $R$  defined with  $\delta_0(I)=I$,  $\delta_1(I)=\sqrt{I}$, $\delta_\star(I)=I+J$ (for some hyperideal $J$ of $R$) and $\delta_*(I)=(I:K)$ (for some hyperideal $K$ of $R$)  for all hyperideals $I$ of $R$, respectively. Also, let $\delta$ be a hyperideal expansion of $R$ and $I,J$ two hyperideals of $R$ such that $I \subseteq J$. Let  $\delta_q:R/I \longrightarrow R/I$ be defined by $\delta_q(J/I)=\delta(J)/I$. Then $\delta_q$ is a hyperideal expansion of $R/I$.
\begin{definition} \cite{ul}
Let $\delta$ be a hyperideal expansion of $R$. A hyperideal $I$ of $R$ is called a $\delta$-primary hyperideal if  $x,y \in R$ and $x \circ y \subseteq I$ imply either $x \in I$ or $y \in \delta(I)$.
\end{definition}
 \begin{definition} \cite{ul}
 Let $f: R_1 \longrightarrow R_2$ be a good hyperring homomorphism, $\delta$ and $\gamma$   hyperideal expansions of $R_1$ and $R_2$, respectively. Then $f$ is called a $\delta \gamma$-homomorphism if $\delta(f^{-1}(I_2)=f^{-1}(\gamma(I_2))$ for each hyperideal $I_2$ of $R_2$.
 \end{definition}
 Moreover, If $f$ is a $\delta \gamma$-epimorphism and $I$ is a hyperideal of $R$ with $Ker f \subseteq I$, then $\gamma(f(I))=f(\delta(I))$. 
\section{ 2-prime hyperideals  }
\begin{definition}  Let $I$ be a proper hyperideal of $R$. We say that $I$ is 2-prime if for
all $x, y \in R$, $x \circ y \subseteq R$  implies $x^2 \subseteq I$ or $y^2 \subseteq I$.
\end{definition}
\begin{example}
Let $(\mathbb{Z},+,.)$ be the ring of integers.  In the multiplicative hyperring $(\mathbb{Z}_A,+,\circ)$ with $A=\{5,7\}$ and the hyperoperation $a \circ b=\{a.r.b \ \vert \ r \in A\}$ for $a, b \in \mathbb{Z}_A$ , the principal hyperideal $3\mathbb{Z}$ is 2-prime.
\end{example}
\begin{theorem} \label{31}
Let $I$ be a hyperideal of $R$. Then
\begin{itemize} 
\item [\rm{(1)}]~ If $I$ is a 2-prime hyperideal, then $\sqrt{I}$ is a prime hyperideal.

\item [\rm{(2)}]~ If $P$ is a prime hyperideal of $R$, then $P^2$ is a 2-prime hyperideal of $R$.

\item [\rm{(3)}]~ Let $R_1$ and $ R_2$ be two multiplicative hyperrings such that $f:R_1 \longrightarrow R_2$ is a good homomorphism. If $I_2$ is a 2-prime hyperideal of $R_2$, then $f^{-1}(I_2)$ is a 2-prime hyperideals of $R_1$.

\item [\rm{(4)}]~Let $R_1$ and $ R_2$ be two multiplicative hyperrings such that $f:R_1 \longrightarrow R_2$ is a good epimorphism. If $I_1$ is a 2-prime hyperideal of $R_1$ with $Ker f \subseteq I_1$, then $f(I_1)$ is a 2-prime hyperideal of $R_2$.

\item [\rm{(5)}]~ If $I$ is a 2-prime hyperideal and $I_1,I_2$ are two subsets of $R$ with $I_1 \circ I_2 \subseteq I$, then
$\bigcup_{x \in  I_1} x^2 \subseteq I$ or $\bigcup_{y \in  I_2} y^2 \subseteq I$.


\end{itemize}
\end{theorem}
\begin{proof}
(1) Let the hyperideal $I$ be 2-prime. If $x \circ y \subseteq \sqrt{I}$ for $x,y \in R$ then for some positive integer $n$, $x^n\circ y^n \subseteq I$. Since hyperideal $I$ is 2-prime, we have $x^{2n} \subseteq I$ or $y^{2n} \subseteq I$. Thus we get $x \in \sqrt{I}$ or $y \in \sqrt{I}$ which means hyperideal $\sqrt{I}$  is prime.

(2) Since $P^2 \subseteq P$, then we are done.

(3) Let $x\circ y \subseteq f^{-1}(I_2)$ for $x,y \in R_1$. Hence $f(x \circ y) \subseteq I_2$. Since the hyperideal $I_2$ is 2-prime, then $f^2(x)=f(x^2) \subseteq I_2$ or $f^2(y)=f(y^2) \subseteq I_2$. Thus $x^2 \subseteq f^{-1}(I_2)$ or $y^2 \subseteq f^{-1}(I_2)$ i.e.,  the hyperideal $f^{-1}(I_2)$ is 2-prime.

(4) Let for some  elements $x_2,y_2 \in R_2$, $x_2 \circ y_2  \subseteq f(I_1)$. Since $f$ is an epimorphism, there exist $x_1,y_1 \in R_1$ such that $f(x_1)=x_2, f(y_1)=y_2$ and so $f(x_1 \circ y_1 )=x_2 \circ y_2 \subseteq f(I_1)$. Now take any $u \in x_1 \circ y_1$. Then we get $f(u) \in f(x_1 \circ y_1) \subseteq f(I_1)$ and so $f(u) = f(w)$ for some $w \in I_1$. This implies that $f(u-w) = 0 $, that is, $u-w \in Ker f \subseteq I_1$ and so $u \in I_1$. Since $I_1$ is a ${\bf C}$-hyperideal of $R_1$, then we conclude that $x_1 \circ y_1 \subseteq I_1$. Since $I_1$ is a 2-prime hyperideal of $R_1$ then $x_1^2  \subseteq I_1$ which implies $x_2^2 =f(x_1^2 ) \subseteq f(I_1)$ or $y_1^2 \subseteq I_1$ which implies $y_2^2=f(y_1^2) \subseteq f( I_1)$.Consequently, $f(I_1)$ is a 2-prime hyperideal of $R_2$.

(5) Let $\bigcup_{x \in  I_1} x^2 \nsubseteq I$ or $\bigcup_{y \in  I_2} y^2 \nsubseteq I$. Hence there exist $x_0\in I_1$ and $y_0 \in I_2$ with $x_0^2 \nsubseteq I$ and $y_0^2 \nsubseteq I$. Since hyperideal $I$ is 2-prime, then $x_0 \circ y_0 \nsubseteq I$ which is a contradiction.
\end{proof}
Let $I$ be a 2-prime hyperideal of $R$. Since $\sqrt{I}=P$ is
a prime hyperideal of R by Theorem  \ref{31} (1), $I$ is referred to as a $P$-2-prime hyperideal of $R$.
\begin{theorem}
Let $I$ be a $P$-2-prime hyperideal of $R$. Then $(I:x^2)$ for all $x \in R-I$ is a   $P$-2-prime hyperideal of $R$.
\end{theorem}
\begin{proof}
Let $y \in (I:x^2)$ for $x \in R-P$. This means that $y \circ x^2 \subseteq I$. Since the hyperideal $I$ is $P$-2-prime and $x \notin P=\sqrt{I}$, then $y^2 \subseteq I$ which means $y \in P$. Therefore we have $ (I:x^2) \subseteq P$. Since $I \subseteq  (I:x^2) \subseteq P$, then $\sqrt{I} \subseteq \sqrt{(I:x^2)} \subseteq \sqrt{P}$. Since $\sqrt{I}=\sqrt{P}=P$, then we have $\sqrt{(I:x^2)}=P$. Assume that for $u,w \in R$, $u \circ w \subseteq (I:x^2)$ such that $w^2 \nsubseteq (I:x^2)$. We have $u \circ w \circ x^2=(u \circ x)\circ (w \circ x) \subseteq I$. Since the hyperideal $I$ is $P$-2-prime and $(w \circ x)^2 \nsubseteq I$, then $(u \circ x)^2=u^2 \circ x^2 \subseteq I$. This means $u^2 \subseteq (I:x^2)$. Consequently, $(I:x^2)$ is a $P$-2-prime hyperideal of $R$. 
\end{proof} \cite{ameri5}
Recall that an element $x \in R$ is called regular if there exists $r \in R$ such that $x \in x^2 \circ r$. So, we can deﬁne that $R$ is regular multiplicative hyperring, if all of elements
in $R$ are regular elements. 
\begin{theorem}
Let $R$ be a regular multiplicative hyperring and $I$ be a 2-prime hyperideal of $R$. Then $I$ is prime.
\end{theorem}
\begin{proof}
 Let $I$ be a 2-prime hyperideal of $R$. Suppose that $a \circ b \subseteq I$ with $a \notin I$ for some $a,b \in R$. Since $R$ is regular, then there exists $r \in R$ such that $a \in a^2 \circ r$. Let $a^2 \subseteq I$. Then $a \in a^2 \circ r \subseteq I$ which is a contradiction. Therefore $a^2 \nsubseteq I$. Since $I$ is a 2-prime hyperideal of $R$, then $b^2 \subseteq I$. This means $b \in \sqrt{I}=I$. Consequently, the hyperideal  $I$ is prime.
\end{proof}
Recall that a proper hyperideal $I$ of $R$ is called semiprime if whenever $x^k\circ y \subseteq I$ for some $x,y \in R$ and $k \in \mathbb{Z}^+$, then $x \circ y \subseteq I$. Note that if $I$ is a semiprime hyperideal of $R$ such that  $J^n \subseteq I$ for some hyperideal $J$ of $R$ and $n \in \mathbb{N}$, then $J \subseteq I$(Proposition 2.4 in \cite{far}). Every prime hyperideal is a semiprime hyperideal, but the converse is not true in general. For example, the hyperideal $<6>$ of $\mathbb{Z}_{30}$ is semiprime, but it is not prime (see Example 2.3 in \cite{far}).

\begin{theorem}
Let $I$ be a semiprime hyperideal of $R$ and $\delta$ be a hyperideal expansion of $R$. Then  $I$ is a 2-prime hyperideal of $R$ if and only if $I$ is a prime hyperideal of $R$.
\end{theorem}
\begin{definition}
A proper hyperideal $I$ of $R$ is called semiprimary if for
all $x,y \in R$ such that $x \circ y \subseteq I$, then either $x$ or $y$ lies in $\sqrt{I}$.
\end{definition} 
\begin{example}
Let $(\mathbb{Z},+,.)$ be the ring of integers.  In the multiplicative hyperring $(\mathbb{Z}_A,+,\circ)$ with $A=\{2,3\}$ and the hyperoperation $a \circ b=\{a.r.b \ \vert \ r \in A\}$ for $a, b \in \mathbb{Z}_A$ , the  hyperideal $11\mathbb{Z}$ is semiprimary.
\end{example}
\begin{theorem}
Every 2-prime hyperideal of $R$ is a semiprimary hyperideal of $R$.
\end{theorem}
\begin{theorem}
Let $\mathbb{Z}_A$ be a multiplicative hyperring of integers and $a(>1)$ be a positive integer  such that each element $x$ of $A$ and $a$ are coprime and  $x^2 \in A$. Then $I=<a>$ is a 2-prime hyperideal of $\mathbb{Z}_A$ if and only if $I=<p^n>$ for some positive integer $n$ and an irreducible   element $p$ of $\mathbb{Z}_A$ or $p = 0$.
\end{theorem}
\begin{proof}
$( \Longrightarrow )$ Let $I=<a>$ be a 2-prime hyperideal of $\mathbb{Z}_A$. Assume that $a$ is not
an irreducible element. Then $a=p_1^{m_1}p_2^{m_2}...p_n^{m_n}$ is a representation of a as a product of distinct prime integers  $p_i$ such that  $m_i$ is a positive integer  for $ 1 \leq i \leq n$. Put $x=p_1^{m_1}$ and $y=p_2^{m_2}...p_n^{m_n}$
by Proposition 4.13 in \cite{das}. For every $c \in A$, $x c y \in <a>$ which implies $x \circ y \subseteq <a>$. Since $<a>$ is 2-prime, then we have $x^2 \subseteq <a>$ or $y^2 \subseteq <a>$. If $x^2 \subseteq <a>$ then for all $d \in A$ we have $x^2d \in <a>$. Thus we obtain $x^2d=p_1^{2m_1}d=\alpha a=\alpha p_1^{m_1}p_2^{m_2}...p_n^{m_n}$ for some $\alpha \in \mathbb{Z}_A$. Since each element of $A$ and $a$ is coprime, then for some $2 \leq i \leq n$ we get $p_i \vert p_1$ which is contradiction. If $y^2 \subseteq <a>$ then for every $e \in A$, $y^2e \in <a>$ which implies $y^2e=(p_2^{m_2}...p_n^{m_n})^2e=\beta a=\beta p_1^{m_1}p_2^{m_2}...p_n^{m_n}$ for $\beta \in \mathbb{Z}$ which means $p_1 \vert p_j$ for some $2 \leq j \leq n$ which is a contradiction.

$(\Longleftarrow) $ Let $I=<p^n>$ for some positive integer $n$ and an irreducible element $p \in \mathbb{Z}_A$. Let $x \circ y \subseteq I$ for some $x,y \in \mathbb{Z}$. Then for $\alpha \in A$, $xy\alpha \in I$ which implies $x=ap^s$ and $y\alpha=bp^t$ for some $a,b \in \mathbb{Z}$ with $n \leq t+s$. If $2s <n$ and $2t <n$, then $2s+2t<2n$ which is a contradiction. Therefore $x^2\beta \in I$ for some $\beta \in A$ or $y^2\alpha^2 \in I$ which implies $x^2 \subseteq I$ or $y^2 \subseteq I$. Consequently, $I$ is a 2-prime hyperideal of $\mathbb{Z}$.
\end{proof}
Let $R$ be a multiplicative hyperring. Then we call $M_n(R)$ as
the set of all hypermatixes of $R$. Also, for all $A = (A_{ij})_{n \times n}, B = (B_{ij})_{n \times n} \in P^\star (M_n(R)), A \subseteq B$ if and only if $A_{ij} \subseteq B_{ij}$\cite{ameri}.  

\begin{theorem} \label{11126} 
Let $R$ be a multiplicative hyperring with scalar
identity 1 and $I$ be a hyperideal of $R$. If $M_n(I)$ is a 2-prime hyperideal of
$M_n(R)$, then $I$ is an 2-prime hyperideal of $R$. 
\end{theorem}
\begin{proof}
Suppose that for $x,y \in R$ , $x \circ y \subseteq I$. Then 
\[ \begin{pmatrix}
x \circ y & 0 & \cdots & 0\\
0 & 0 & \cdots & 0\\
\vdots& \vdots & \ddots \vdots\\
0 & 0 & \cdots & 0
\end{pmatrix}
\subseteq M_n(I). \]
It is clear that 
\[ \begin{pmatrix}
x \circ y & 0 & \cdots & 0\\
0 & 0 & \cdots & 0\\
\vdots& \vdots & \ddots \vdots\\
0 & 0 & \cdots & 0
\end{pmatrix}
=
\begin{pmatrix}
x & 0 & \cdots & 0\\
0 & 0 & \cdots & 0\\
\vdots& \vdots & \ddots \vdots\\
0 & 0 & \cdots & 0
\end{pmatrix}
\begin{pmatrix}
y & 0 & \cdots & 0\\
0 & 0 & \cdots & 0\\
\vdots& \vdots & \ddots \vdots\\
0 & 0 & \cdots & 0
\end{pmatrix}
.\]
Since $M_n(I)$ is a 2-prime hyperideal of $M_n(R)$ then
\[ \begin{pmatrix}
x & 0 & \cdots & 0\\
0 & 0 & \cdots & 0\\
\vdots& \vdots & \ddots \vdots\\
0 & 0 & \cdots & 0
\end{pmatrix}^2
=\begin{pmatrix}
x^2 & 0 & \cdots & 0\\
0 & 0 & \cdots & 0\\
\vdots& \vdots & \ddots \vdots\\
0 & 0 & \cdots & 0
\end{pmatrix} \subseteq  M_n(I)\]
which means $x^2 \subseteq I$
or
\[ \begin{pmatrix}
y & 0 & \cdots & 0\\
0 & 0 & \cdots & 0\\
\vdots& \vdots & \ddots \vdots\\
0 & 0 & \cdots & 0
\end{pmatrix}^2
=\begin{pmatrix}
y^2 & 0 & \cdots & 0\\
0 & 0 & \cdots & 0\\
\vdots& \vdots & \ddots \vdots\\
0 & 0 & \cdots & 0
\end{pmatrix} \subseteq  M_n(I)\]
which means $y^2 \subseteq I$. Therefore $I$ is a 2-prime hyperideal of $R$.
\end{proof}
Let $(R, +, \circ)$ be a hyperring. We define the relation $\gamma$ as follows:\\ 
$a \gamma b$ if and only if $\{a,b\} \subseteq U$ where $U$ is a finite sum of finite products of
elements of R, i.e.,\\
 $a \gamma b \Longleftrightarrow \exists z_1, ... , z_n \in R$ such that $\{a, b\} \subseteq \sum_{j \in J} \prod_{i \in I_j} z_i; \ \ I_j, J \subseteq \{1,... , n\}$.

We denote the transitive closure of $\gamma$ by $\gamma ^{\ast}$. The relation $\gamma ^{\ast}$ is the smallest equivalence relation on a multiplicative hyperring $(R, +, \circ)$ such that the
quotient $R/\gamma ^{\ast}$, the set of all equivalence classes, is a fundamental ring. Let $\mathfrak{U}$
be the set of all finite sums of products of elements of R we can rewrite the
definition of $\gamma ^{\ast}$ on $R$ as follows:\\
$a\gamma ^{\ast}b$ if and only if there exist  $z_1, ... , z_n \in R$ with $z_1 = a, z_{n+1 }= b$ and $u_1, ... , u_n \in \mathfrak{U}$ such that
$\{z_i, z_{i+1}\} \subseteq u_i$ for $i \in \{1, ... , n\}$.
Suppose that $\gamma ^{\ast}(a)$ is the equivalence class containing $a \in R$. Then, both
the sum $\oplus$ and the product $\odot$ in $R/\gamma ^{\ast}$ are defined as follows:$\gamma ^{\ast}(a) \oplus \gamma ^{\ast}(b)=\gamma ^{\ast}(c)$ for all $c \in \gamma ^{\ast}(a) + \gamma ^{\ast}(b)$ and $\gamma ^{\ast}(a) \odot \gamma ^{\ast}(b)=\gamma ^{\ast}(d)$ for all $d \in \gamma ^{\ast}(a) \circ \gamma ^{\ast}(b)$
Then $R/\gamma ^{\ast}$ is a ring, which is called a fundamental ring of $R$ (see also \cite{f9}).
\begin{theorem}
Let $R$ be a multiplicative hyperring with scalar identity $1$. Then the hyperideal $I$ of $R$ is 2-prime if and only if $I/\gamma ^{\ast}$ is a 2-prime ideal of $R/\gamma ^{\ast}$. 
\end{theorem}
\begin{proof}
($\Longrightarrow$) Let for $x, y \in R/\gamma ^{\ast}, \ x \odot y \in I/\gamma ^{\ast}$. Thus, there exist $a, b\in R$ such
that $x =\gamma^{\ast}(a), y = \gamma^{\ast}(b)$ and $x \odot y = \gamma^{\ast}(a) \odot \gamma^{\ast}(b) =\gamma^{\ast}(a \circ b)$. So, $\gamma^{\ast} (a) \odot \gamma^{\ast}(b)=
\gamma^{\ast}(a \circ b) \in I/\gamma^{\ast}$, then $a \circ b \subseteq I$. Since $I$ is 2-prime hyperideal, then $a^2 \subseteq I$ or $b^2 \subseteq I$. Hence $x^2= \gamma^{\ast}(a)^2 \in I/\gamma ^{\ast}$ or $y^2= \gamma^{\ast}(b)^2 \in I/\gamma ^{\ast}$.
Thus $I/\gamma ^{\ast}$ is a 2-prime ideal of $R/\gamma ^{\ast}$.\\
($\Longleftarrow$) Suppose that $a \circ b \subseteq I$ for $a, b \in R$, then $\gamma^{\ast}(a), \gamma^{\ast}(b)\in R/\gamma^{\ast} $ and
$\gamma^{\ast}(a) \odot\gamma^{\ast}(b) = \gamma^{\ast}(a \circ b) \in I/\gamma^{\ast}$. Since $I/\gamma ^{\ast}$ is a 2-prime ideal of $R/\gamma ^{\ast}$,  then we have $\gamma^{\ast}(a)^2 \in I/\gamma^{\ast}$ or $\gamma^{\ast}(b)^2 \in I/\gamma^{\ast}$. It means that $a^2  \subseteq I$ or $b^2 \subseteq I$. Hence $I$ is a 2-prime hyperideal of $R$.
\end{proof}
\begin{lem} \label{1400}
Let $R$ be a local multiplicative hyperring with maximal hyperideal $M$ and $P$ be a prime hyperideal of $R$. Then the hyperideal $P\circ M$ is 2-prime. Moreover, if  $P \circ M$ is a prime hyperideal of $R$ then $P\circ M=P$. 
\end{lem}
\begin{proof}
Let $x \circ y \subseteq P \circ M \subseteq P$ for some $x, y \in R$. This means $x \in P$ or $y \in P$. Suppose that $x \in P$. Clearly, $x$ is not unit. Then we have $x \in M$ which implies $x^2 \subseteq P \circ M$. Thus the hyperideal $P\circ M$ is 2-prime.

For the second assertion, assume that the hyperideal $P\circ M$ is prime. Suppose that $x \in P$. Since $P \subseteq M$, then $x^2 \subseteq P \circ M$ which implies $x \in P \circ M$, because $P \circ M$ is a prime hyperideal of $R$. This means $P \subseteq P \circ M$. Since $P \circ M \subseteq P$, then we have $P \circ M=P$.
\end{proof}
\begin{theorem} \label{1401}
Let $R$ be a local multiplicative hyperring with maximal hyperideal $M$. Then every 2-prime hyperideal of $R$ is prime if and only if  for each minimal prime hyperideal $P$ over an arbitrary 2-prime hyperideal $I$, $I \circ M=P$. Furthermore, If every 2-prime hyperideal of $R$ is prime then $M$ is an idempotent hyperideal.  
\end{theorem}
\begin{proof}
$(\Longrightarrow) $ Assume that  every 2-prime hyperideal of $R$ is prime. Since the hyperideal $I$ is 2-prime, then $I=P$ is prime. Then $I \circ M$ is a 2-prime hyperideal of $R$, by Lemma \ref{1400}. This means $I \circ M$ is a prime hyperideal of $R$. Then we conclude that $I\circ M=I$, by Lemma \ref{1400}. 

$(\Longleftarrow)$ Suppose that  $I$ is a 2-prime hyperideal of $R$. Then  we have $I \subseteq \sqrt{I}=P$, by theorem \ref{31}(1). By the assumption, we get $P=I \circ M \subseteq I \cap M=I$. This implies that $I=P$ is a prime hyperideal of $R$.
\end{proof}
\begin{theorem}
Let every 2-prime hyperideal of $R$ be prime and $P$ be an arbitrary prime hyperideal of $R$. Then $P^2=P$.
\end{theorem}
\begin{proof}
Assume that  every 2-prime hyperideal of $R$ is prime and $P$ is an arbitrary prime hyperideal of $R$. Then the hyperideal $P^2$ is 2-prime, by Theorem \ref{31}(2). Since very 2-prime hyperideal of $R$ is prime, then the hyperideal $P^2$ is prime. It is easy to see that $P^2$ is equal to $P$.
\end{proof}
\begin{definition}
Let $I$ be a hyperideal of $R$. We say that a 2-prime hyperideal $P$ is minimal over $I$ if there is no a 2-prime hyperideal $P^\prime$ of $R$ with $I \subseteq P^\prime  \subset P$. Note that 2-$Min_R(I)$ denotes the set of minimal 2-prime hyperideals over $I$.
\end{definition}
\begin{theorem} \label{1402}
Let $R$ be a local multiplicative hyperring with maximal hyperideal $M$ and $P$ be a prime hyperideal of $R$. Let for each 2-prime hyperideal $I$ of $R$, $(\sqrt{I})^2 \subseteq I$. Then the followings are equivalent:
\begin{itemize}
\item[\rm{(1)}]~If $P \in Min_R(I)$ for each hyperideal $I \in 2-Min_R(P^2)$, then $I \circ M=P$.
\item[\rm{(2)}]~ If $I \subseteq P$  for each hyperideal $I \in 2-Min_R(P^2)$, then $I=P$.
\end{itemize}
\end{theorem}
\begin{proof}
(1) $\Longrightarrow$ (2) Suppose that $I \subseteq P$ for some $I \in 2-Min_R(P^2)$. First, we show that the hyperideal $P$ is minimal over $I$. Assume that we have $I \subseteq P^\prime \subseteq P$ for some prime hyperideal $P^\prime $ of $R$. By the assumption, we get $P^2 \subseteq I \subseteq P^\prime \subseteq P$. Let $a \in P$. Then we obtain $a^2 \subseteq P^2$ which implies $a^2 \subseteq P^\prime$. This means $a \in P^\prime$, since $P^\prime$ is a prime hyperideal of $R$. Thus we conclude that   the hyperideal $P$ is minimal over $I$. Now, since $I\circ M \subseteq I \subseteq P$ and $I \circ M =P$, then we get $I=P$. 

(2) $\Longrightarrow$ (1) Let $P \in Min_R(I)$ such that $I \in 2-Min_R(P^2)$. We get $\sqrt{I}=P$ since the hyperideal $\sqrt{I}$ is prime. Clearly, $(\sqrt{I})^2=P^2\subseteq I \subseteq P$. By Lemma \ref{1400}, $P \circ M$ is a 2-prime hyperideal of $R$. Since $P^2 \subseteq P \circ M \subseteq I$ and $I=P$ , then $P \circ M=I \circ M=P$. 
\end{proof}
\begin{corollary}
Let $R$ be a local multiplicative hyperring with maximal hyperideal $M$ and $P$ be a prime hyperideal of $R$. Let for each $P$-2-prime hyperideal  $I$, $I \in 2-Min_R(P^2)$. Then for each hyperideal $I \in 2-Min_R(P^2)$ with $I \subseteq P$, $I=P$ if and only if every 2-prime hyperideal of $R$ is prime.
\end{corollary}
\begin{proof}
The claim follows by Theorem \ref{1401} and Theorem \ref{1402}.
\end{proof}

\section{expansion of 2-prime hyperideals}
\begin{definition}
Let  $\delta$ be a hyperideal expansion  of $R$. A proper hyperideal $I$ of $R$ is called  $\delta$-2-primary if for $a, b \in R$, $a \circ b \subseteq I$ implies either $a^2 \subseteq I$ or $b^2 \subseteq \delta(I)$.
\end{definition}
\begin{example}
Suppose that $(\mathbb{Z},+, \cdot)$ is the ring of integers. Assume that $(\mathbb{Z},+, \circ)$ is a multiplicative hyperring with a hyperoperation $a \circ b$ for $a,b \in \mathbb{Z}$. Define the hyperideal expansion $\delta_\star$ by $\delta_\star(I)=3\mathbb{Z}+I$. Since $\delta_\star(2\mathbb{Z})=3\mathbb{Z}+2\mathbb{Z}=\mathbb{Z}$ then we conclude that  $I=2 \mathbb{Z}$ is a $\delta_\star$-2-primary hyperideal of $(\mathbb{Z},+, \circ)$.
\end{example}
Clearly, every prime hyperideal of $R$ is a $\delta$-2-primary hyperideal but its inverse is not true in general.
\begin{example}
Let $\mathbb{Z}$ be the ring of integers, $E$ be the set of all even integers of $\mathbb{Z}$ and $A$  be the set of all positive even integers of $\mathbb{Z}$. In the multiplicative hyperring $\mathbb{Z}_A$ (See Example 3.5 in \cite{das}) , the hyperideal  $E$ is $\delta_1$-2-primary but is not prime.
\end{example}
\begin{example}
In the multiplicative hyperring $(\mathbb{Z},+,\circ)$ with the hyperoperation $a \circ b =\{a.b\}$ for $a,b \in \mathbb{Z}$, the hyperideal $2 \mathbb{Z}$ is $\delta_1$-2-primary.
\end{example}
We start the section with the following trivial result, and hence we omit its proof.
\begin{theorem}
Let $I$ be a proper of $R$. Then
\begin{itemize}
\item[\rm{(1)}]~ $I$ is a $\delta_0$-2-primary hyperideal if and only if $I$ is a 2-prime hyperideal.
\item[\rm{(2)}]~ If $I$ is a primary hyperideal, then $I$ is a $\delta_1$-2-primary hyperideal.
\item[\rm{(3)}]~ If $I$ is a 2-prime hyperideal, then $I$ is a $\delta$-2-primary hyperideal for every hyperideal expansion $\delta$ of $R$.
\item[\rm{(4)}]~ If $I$ is a $\delta$-primary hyperideal, then $I$ is a $\delta$-2-primary hyperideal for every hyperideal expansion $\delta$ of $R$.
\item[\rm{(5)}]~If $I$ is a $\delta$-2-primary hyperideal of $R$ such that $\delta(I) \subseteq \gamma (I)$ for some hyperideal expansion  $\gamma$ of $R$, then $I$ is a $\gamma$-2-primary hyperideal of $R$. 
\end{itemize} 
\end{theorem}

\begin{theorem}
Let $I$ be a  hyperideal of $R$ and $\delta$ a hypoerideal expansion of $R$. If $I$ is a $\delta$-2-primary hyperideal of $R$, then for some idempotent element $a \in R-I$, $(I:a)$ is a $\delta$-2-primary hyperideal of $R$.
\end{theorem}
\begin{proof}
Let $x \circ y \subseteq (I:a)$ such that $x^2 \nsubseteq (I:a) = (I:a^2)$ for some $x,y \in R$. This means that $x\circ y \circ a \subseteq I$ but $x^2 \circ a^2 \nsubseteq I$. Since $I$ is a $\delta$-2-primary hyperideal of $R$, we get $y^2 \subseteq \delta(I) \subseteq \delta(I:a)$. Thus $(I:a)$ is a $\delta$-2-primary hyperideal of $R$.
\end{proof}
\begin{theorem}
Let $I$ be a hyperideal of $R$ such that for each $a \in R-I$, $(I:a)=(I:a^2)$ and $\delta$ a hyperideal expansion of $R$. If the hyperideal $I$ is irreducible, then $I$ is a $\delta$-2-primary hyperideal of $R$.
\end{theorem}
\begin{proof}
We suppose that $I$ is not a $\delta$-2-primary hyperideal of $R$ and look for a contradiction. This means that there exist $x,y \in R$ such that $x \circ y \subseteq I$ but $x^2 \nsubseteq I$ and $y^2 \nsubseteq \delta(I)$. Thus we get $x \notin I$ and $y \notin \delta(I)$. Let $t \in (I+<x>) \cap  (I+<y>)$. Then there are $a_1,a_2 \in I$ and $r_1,r_2 \in R$ such that for some $t_1 \in r_1 \circ x$ and $t_2 \in r_2 \circ y$ we have $t=a_1+t_1=a_2+t_2$. So we get $y \circ (a_2+t_2) \subseteq y \circ a_2 + y \circ t_2$. Also, we have $y \circ (a_1+t_1) \subseteq y \circ a_1 +y \circ t_1$. Since $y \circ a_1 +y \circ t_1 \subseteq I$ then $y \circ (a_2+t_2) \subseteq I$. This implies that $ (y \circ a_2 +y \circ t_2) \cap I \neq \varnothing$. Since $I$ is a $\mathbf{C}$-hyperideal of $R$, then $y \circ a_2 +y \circ t_2 \subseteq I$. Hence $y \circ t_2 \subseteq I$ which implies $r_2 \circ y^2 \subseteq I$. Therefore $r_2 \in (I:y^2)$ which means $r_2 \in (I:y)$, by the assumption. So $r_2 \circ y \subseteq I$. Thus we have $t=a_2+t_2 \subseteq a_2+r_2 \circ y \subseteq I$. Then $(I+<x>) \cap  (I+<y>) \subseteq I$. Since $I \subseteq (I+<x>) \cap  (I+<y>)$ then we obtain $I=(I+<x>) \cap  (I+<y>)$. This is a contradiction since $I$ is irreducible. Consequently, the hyperideal $I$ of $R$ is $\delta$-2-primary.
\end{proof}
\begin{theorem}
 Let $\delta$ be a hyperideal expansion of $R$. Let  $I$ and $\delta(I)$ be  semiprime hyperideals of $R$. Then $I$ is a $\delta$-2-primary hyperideal of $R$ if and only if $I$ is $\delta$-primary. 
\end{theorem}
\begin{theorem}
Let $I$ be a hyperideal of $R$ and $\delta$ a hyperideal expansion of $R$ such that $\sqrt{\delta(I)} \subseteq \delta(\sqrt{I})$. If $I$ is a $\delta$-2-primary hyperideal of $R$,  then $\sqrt{I}$ is a $\delta$-primary hyperideal of $R$.
\end{theorem}
\begin{proof}
Let $x \circ y \subseteq \sqrt{I}$ for some $x,y \in R$ such that $x \notin \sqrt{I}$. This means that we have $x^n\circ y^n \subseteq I$ for some $n \in \mathbb{N}$. Clearly, $x^{2n} \nsubseteq I$. Since $I$ is a $\delta$-2-primary hyperideal of $R$, then $y^{2n} \subseteq  \delta(I)$ which implies $y \in \sqrt{\delta(I)} \subseteq \delta(\sqrt{I})$. This means that $\sqrt{I}$ is a $\delta$-primary hyperideal of $R$.
\end{proof}
\begin{theorem}
Let $I$ be a hyperideal of $R$ and $\delta$ a hyperideal expansion of $R$ such that $\sqrt{\delta(I)} = \delta(I)$. If the hyperideal $I$ is primary, then $I$ is a $\delta$-2-primary hyperideal of $R$.
\end{theorem}
\begin{proof}
Let $x \circ y \subseteq I$ for some $x,y \in R$. We have $x \in I$ or $y \in \sqrt{I} \subseteq \sqrt{\delta(I)}$. Thus $x^2 \subseteq I$ or $y^2 \subseteq \sqrt{\delta(I)}=\delta(I)$. Consequently, $I$ is a $\delta$-2-primary hyperideal of $R$.
\end{proof}
Recall from \cite{ul} that a hyperideal expansion $\delta$ of $R$ has the property of intersection preserving, if it satisﬁes  $\delta(I \cap J)=\delta(I) \cap \delta(J)$ for any hyperideals $I$ and $J$ of $R$.
\begin{theorem}
Let $\delta$ has the property of intersection preserving. If $I_i$ is a $\delta$-2-primary hyperideal of $R$ with $\delta(I_i)=P$ for all $1 \leq i \leq n$. Then $I=\bigcap_{i=1}^nI_i$ is a $\delta$-2-primary hyperideal of $R$. 
\end{theorem}
\begin{proof}
Let $a , b \in I$ with $a \circ b \subseteq I$ and $a^2 \nsubseteq I$. Then there exists some $1 \leq t \leq n$ such that $a^2 \nsubseteq I_t$. Since $I_t$ is a $\delta$-2-primary hyperideal of $R$ then $b^2 \subseteq \delta(I_t)=P$. Therefore $b^2 \subseteq P=\bigcap_{i=1}^n\delta(I_i)=\delta(\bigcap_{i=1}^nI_i)$. Thus we conclude that $I=\bigcap_{i=1}^nI_i$ is a $\delta$-2-primary hyperideal of $R$.
\end{proof}
\begin{theorem}
If  $\{I_j \ \vert \ j \in \Lambda\}$ is a directed set of $\delta$-2-primary hyperideals of $R$, then $\bigcup_{j \in \Lambda}I_j$ is a $\delta$-2-primary hyperideal of $R$.
\end{theorem}
\begin{proof}
Let $x \circ y \subseteq \bigcup_{j \in \Lambda}I_j$ for some $x,y \in R$. Assume that $x^2 \nsubseteq \bigcup_{j \in \Lambda}I_j$. This implies that there exists $i \in \Lambda$ such that $x^2 \nsubseteq I_i$. Since $I_i$ is a $\delta$-2-primary hyperideal of $R$, then $y ^2 \subseteq \delta(I_i) \subseteq \delta(\bigcup_{j \in \Lambda}I_j)$. Thus $\bigcup_{j \in \Lambda}I_j$ is a $\delta$-2-primary hyperideal of $R$.
\end{proof}
\begin{theorem}
Let $R$ be a regular multiplicative hyperring and $\delta$   be  hyperideal expansion of $R$. If $I$ is a $\delta$-2-primary hyperideal of $R$ then $I$ is $\delta$-primary. 
\end{theorem}
\begin{proof}
Let $I$ is a $\delta$-2-primary hyperideal of $R$. Assume that  $a,b \in R$ and $a \circ b \subseteq I$ such that $a \notin I$. Since $R$ is regular, then there exists $r \in R$ such that $a \in a^2 \circ r$. Let $a^2 \subseteq I$. Then $a \in a^2 \circ r \subseteq I$ which is a contradiction. So $a^2 \nsubseteq I$. Since $I$ is a $\delta$-2-primary hyperideal of $R$, then $b ^2 \subseteq \delta(I)$. Since $R$ is regular, then there exists $r^\prime \in R$ such that $b \in b^2 \circ r^\prime \subseteq \delta(I)$. Therefore,  the hyperideal $I$ is $\delta$-primary.
\end{proof}
\begin{theorem} \label{homo}
Let $\delta$ and $\gamma$ be  hyperideal expansions of $R_1$ and $R_2$, respectively and  $f:R_1 \longrightarrow R_2$  a $\delta \gamma$-homomorphism. Then:
\begin{itemize}
\item[\rm{(1)}]~ If $I_2$ is a $\gamma$-2-primary hyperideal of $R_2$, then $f^{-1}(I_2)$ is a $\delta$-2-primary hyperideal of $R_1$.
\item[\rm{(2)}]~ Let $I_1$ be a hyperideal of $R_1$ and $f$  an epimorphism with $Ker f \subseteq I_1$.  Then $f(I_1)$ is a $\gamma$-2-primary hyperideal of $R_2$ if and only if  $I_1$ is a $\delta$-2-primary hyperideal of $R_1$.
\end{itemize}
\end{theorem}
\begin{proof}
(1) Let $a \circ b \subseteq f^{-1}(I_2)$ for some $a,b \in R_1$. Then we have $f(a \circ b)=f(a) \circ f(b) \subseteq I_2$. Since $I_2$ is a $\gamma$-2-primary hyperideal of $R_2$, we get $(f(a))^2 \subseteq I_2$ or $(f(b))^2 \subseteq \gamma(I_2)$ which implies $f(a^2) \subseteq I_2$ or $f(b^2) \subseteq \gamma(I_2)$.  Then  $a^2 \subseteq f^{-1}(I_2)$ or $b^2 \subseteq f^{-1}(\gamma(I_2))=\delta(f^{-1}(I_2))$, because  $f$  is a $\delta \gamma$-homomorphism. Consequenlty, $f^{-1}(I_2)$ is a $\delta$-2-primary hyperideal of $R_1$.

(2) ($\Longrightarrow$) It is quite clear from (1).\\
($\Longleftarrow$) Let $a_2 \circ b_2 \subseteq f(I_1)$ for some $a_2,b_2 \in R_2$. Then for some $a_1,b_1 \in R_1$ we have $f(a_1)=a_2$ and $f(b_1)=b_2$. So $f(a_1) \circ f(b_1)=f(a_1 \circ b_1) \subseteq f(I_1)$. Now, take any $u \in a_1 \circ b_1$. Then $f(u) \in f(a_1 \circ b_1) \subseteq f(I_1)$ and so there exists $w \in I_1$ such that $f(u)=f(w)$. This means that $f(u-w)=0$, that is, $u-w \in Ker f \subseteq I_1$ and then $u \in I_1$. Since  $I_1$ is a {\bf C}-hyperideal of $R_1$, then we get $a_1 \circ b_1 \subseteq  I_1$. Since $I_1$ is a 2-primary  hyperideal of $R_1$, then we obtain $a_1^2 \subseteq I_1$ or $b_1^2 \subseteq \delta(I_1)$. This implies that $f(a_1^2)=a_2^2 \subseteq f(I_1)$ or $f(b_1^2)=b_2^2 \subseteq f(\delta(I_1))=\gamma(f(I_1))$. Thus $f(I_1)$ is a $\gamma$-2-primary hyperideal of $R_2$.
\end{proof}
\begin{corollary}
Let $I$ and $J$ be two hyperideals of $R$ such that $I \subseteq J$. Then $J$ is a $\delta$-primary hyperideal of $R$ if and only if $J/I$ is a $\delta_q$-primary hyperideal of $R/I$.
\end{corollary}
\begin{proof}
The claim is verified from Theorem \ref{homo}.
\end{proof}
\begin{corollary}
Let $I$ be a hyperideal of $R$ and $S$ a subhyperring of $R$ with $S \nsubseteq I$. If $I$ is a $\delta$-2-primary hyperideal of $R$, then $I \cap S$ is a $\delta$-2-primary hyperideal of $S$. 
\end{corollary}
\begin{theorem}
Let $\delta$ be a hyperideal expansion of $R$. Then the followings are equivalent:
\begin{itemize}
\item[\rm{(1)}]~Every proper principal hyperideal of $R$ is  $\delta$-2-primary. 
\item[\rm{(2)}]~ Every proper hyperideal of $R$ is  $\delta$-2-primary. 
\end{itemize}
\end{theorem}
\begin{proof}
$(1) \Longrightarrow (2)$ Assume that $I$ is a proper hyperideal of $R$ such that $x \circ y \subseteq I$ for some $x,y \in R$. Take any $t \in x \circ y$. Clearly, $x \circ y \cap <t> \neq \varnothing$. This implies that $x \circ y \subseteq <t> \subseteq I$. Since $<t>$ is a $\delta$-2-primary hyperideal of $R$, then we get $x^2 \subseteq <t> \subseteq I$ or $y^2 \subseteq \delta(<t>) \subseteq \delta(I)$. Therefore $I$ is a $\delta$-2-primary hyperideal of $R$. 

$(2) \Longrightarrow (1)$ It is clear.
\end{proof}

Let $(R_1,+_1,\circ_1)$ and $(R_2,+_2,\circ_2)$ be two multiplicative hyperrings with non zero identity. \cite{ul} Recall $(R_1 \times R_2, +,\circ)$ is a multiplicative hyperring with the operation $+$ and the hyperoperation $\circ$ are defined respectively as

$(x_1,x_2)+(y_1,y_2)=(x_1+_1y_1,x_2+_2y_2)$ 
and

$(x_1,x_2) \circ (y_1,y_2)=\{(x,y) \in R_1 \times R_2 \ \vert \ x \in x_1 \circ_1 y_1, y \in x_2 \circ_2 y_2\}$. 

Assume that $\delta_1$ and $\delta_2$ are hyperideal expansions of $R_1$ and $R_2$ respectively. If $\delta_{R_1 \times R_2}$ is a function of hyperideals of $R$ with $\delta_{R_1 \times R_2}(I_1 \times I_2)=\delta_1(I_1) \times \delta_2(I_2)$ for every hyperideals $I_1$ and $I_2$ of $R_1$ and $R_2$, respectively, then $\delta_{R_1 \times R_2}$ is a hyperideal expansion of ${R_1 \times R_2}$.
\begin{theorem} \label{times}
Let $(R_1, +_1,\circ 1)$ and $(R_2,+_2,\circ_2)$ be two multiplicative hyperrings with non zero identity such that $\delta_1$ and $\delta_2$ be hyperideal expansions of $R_1$ and $R_2$, respectively. Let $I_1$ be a hyperideal of $R_1$. Then the hyperideal $I_1$ of $R_1$ is $\delta_1$-2-primary if and only if $I_1 \times R_2$ is a $\delta_{R_1 \times R_2}$-2-primary of $R_1 \times R_2$.
\end{theorem}
\begin{proof}
($\Longrightarrow$) Let $(x_1,x_2) \circ (y_1,y_2) \subseteq I_1 \times R_1$ for some $(x_1,x_2), (y_1,y_2) \in R_1 \times R_2$. This means  $x_1 \circ_1 y_1 \subseteq I_1$. Since $I_1$ is a $\delta_1$-2-primary hyperideal of $R_1$, then we get $x_1^2 \subseteq I_1$ or $y_1^2 \subseteq \delta_1(I_1)$. This implies that $(x_1,x_2)^2=(x_1^2,x_2^2) \subseteq I_1\times R_2$ or $(y_1,y_2)^2=(y_1^2,y_2^2) \subseteq \delta_{R_1 \times R_2}(I_1 \times R_2)$.

($\Longleftarrow$) Assume on the contrary that $I_1$ is not a $\delta_1$-2-primary hyperideal of $R_1$. So $x_1 \circ_1 y_1 \subseteq I_1$ with $x_1, y_1 \in R_1$ implies that $x_1^2 \nsubseteq I_1$ and $y_1^2 \nsubseteq \delta_1(I_1)$. It is clear that $(x_1,1_{R_2})\circ (y_1, 1_{R_2}) \subseteq I_1 \times R_2$. Since $I_1 \times R_2$ is a $\delta_{R_1 \times R_2}$-2-primary hyperideal of $R_1 \times R_2$, then we have $(x_1,1_{R_2})^2=(x_1^2, 1_{R_2}) \subseteq I_1 \times R_2$ or $(y_1,1_{R_2})^2=(y_1^2,1_{R_2}) \subseteq \delta_{R_1 \times R_2}(I_1 \times R_2)$. Hence we get $x_1^2 \subseteq I_1$ or  $y_1^2 \subseteq \delta_1(I_1)$ which is a contradiction. Consequently, $I_1$ is a $\delta_1$-2-primary hyperideal of $R_1$.  
\end{proof}
\begin{theorem}
Let $(R_1, +_1,\circ 1)$ and $(R_2,+_2,\circ_2)$ be two multiplicative hyperrings with non zero identity such that $\delta_1$ and $\delta_2$ be hyperideal expansions of $R_1$ and $R_2$, respectively. Let $I_1$ and $I_2$ be  some hyperideals of $R_1$ and $R_2$, respectively. Then the
following statements are equivalent:
\begin{itemize}
\item[\rm{(1)}]~ $I_1 \times I_2$ is a $\delta_{R_1 \times R_2}$-2-primary hyperideal of $R_1 \times R_2$.
\item[\rm{(2)}]~ $I_1=R_1$ and $I_2$ is a $\delta_2$-2-primary hyperideal of $R_2$ or $I_2=R_2$ and $I_1$ is a $\delta_1$-2-primary hyperideal of $R_1$.
\end{itemize}
\end{theorem}
\begin{proof}
(1) $\Longrightarrow$ (2) Assume that $I_1=R_1$. Then $I_2$ is a $\delta_2$-2-primary hyperideal of $R_2$, by Theorem \ref{times}.

(2) $\Longrightarrow$ (1) This can be proved, by using Theorem \ref{times}.
\end{proof}
\begin{example}
Suppose that $(\mathbb{Z},+,.)$ is the ring of integers. Then $(\mathbb{Z},+,\circ_1)$ is a multiplicative hyperring with
a hyperoperation $a \circ_1 b =\{ab,7ab\}$. Also, $(\mathbb{Z},+,\circ_2)$ is a multiplicative hyperring with
a hyperoperation $a \circ_2 b =\{ab,5ab\}$. Note that $(\mathbb{Z} \times \mathbb{Z},+,\circ)$ is a multiplicative hyperring with a hyperoperation $(a,b) \circ (c,d)=\{(x,y) \in \mathbb{Z} \times \mathbb{Z} \ \vert \ x \in a \circ_1 c , y \in b \circ_2
d \}$. Clearly, $7\mathbb{Z}=\{7t \ \vert \ t \in \mathbb{Z}\}$ and $5\mathbb{Z}=\{5t \ \vert \ t \in \mathbb{Z}\}$ are two $\delta_0$-2-primary of $(\mathbb{Z},+,\circ_1)$ and $(\mathbb{Z},+,\circ_2)$, respectively. Since $(5,0) \circ (0,7) \subseteq 7\mathbb{Z} \times 5 \mathbb{Z}
$ but $(5,0)^2,(0,7)^2 \nsubseteq 7\mathbb{Z} \times 5\mathbb{Z}$ and  $(5,0)^2,(0,7)^2 \nsubseteq \delta_{\mathbb{Z} \times \mathbb{Z}}(7\mathbb{Z} \times 5\mathbb{Z})=7\mathbb{Z} \times 5\mathbb{Z}$, then $7\mathbb{Z} \times 5\mathbb{Z}$ is not a $\delta_0 \times \delta_0$-2-primary hyperideal of $\mathbb{Z}  \times \mathbb{Z}$.
\end{example}


\end{document}